\documentclass[a4paper,12pt]{article}
\usepackage{amsmath,amsthm, amssymb }
\usepackage{graphicx}
\usepackage{graphics}

\newcommand{\Addresses}{{
  \bigskip
  \footnotesize

  \textsc{Department of Mathematics, Saarland University, P.O. Box 151150,  Saar- br{\"u}cken 66041, Germany} and
  \textsc{ Faculty of Mathematics and Mechanics, St. Petersburg State University, Universitetskii pr. 28,  St. Petersburg 198504, Russia}\par\nopagebreak
  \textit{E-mail address:} \texttt{darya@math.uni-sb.de}

  \bigskip

\textsc{St. Petersburg Department of Steklov Institute, Fontanka 27, St. Petersburg 191023, Russia} and
  \textsc{Faculty of Mathematics and Mechanics, St. Petersburg State University, Universitetskii pr. 28,  St. Petersburg 198504, Russia}\par\nopagebreak
  \textit{E-mail address:} \texttt{al.il.nazarov@gmail.com}
}}

\newtheorem{theorem}{Theorem}

\theoremstyle{definition}
\newtheorem{definition}[theorem]{Definition}

\theoremstyle{definition}
\newtheorem{remark}{Remark}

\theoremstyle{plain}
\newtheorem{thm}{Theorem}[section]
\newtheorem{lemma}[thm]{Lemma}

\author{D.\,E. Apushkinskaya, A.\,I. Nazarov}
\title{A counterexample to the Hopf-Oleinik lemma \\
(elliptic case) }

\begin{document}
\maketitle

\begin{abstract}
We construct a new counterexample to the Hopf-Oleinik boundary point lemma. It shows that for convex domains $C^{1,\, \text{Dini}}$ assumption on $\partial\Omega$ is the necessary and sufficient condition providing the Hopf-Oleinik type estimates.
\end{abstract}
\vspace{0.3cm}

 \begin{flushright}
\textit{Dedicated to Professor M.V. Safonov}
\end{flushright}
\bibliographystyle{amsalpha}
\newcommand{\tg}{\textup{tg}}
\newcommand{\ep}{\varepsilon}
\newcommand{\osc}[1]{\underset{#1}{\textup{osc}}}

\section{Introduction}
The influence of the properties  of a domain to the behavior of a solution is one of the  most important topic  in the qualitative analysis of partial differential equations. 

The significant result in this field is the Hopf-Oleinik lemma, known also as the "Boundary Point Principle". This celebrated lemma states: \vspace{0.1cm}

\noindent
\textit{Let $u$ be a nonconstant solution to a second-order homogeneous uniformly elliptic nondivergence equation with bounded measurable coefficients, and let $u$ attend its extremum at a point $x^0$ located on the boundary of a domain $\Omega \subset \mathbb{R}^n$. Then $\frac{\partial u}{\partial \mathbf{n}}(x^0)$ is necessarily nonzero provided that  $\partial\Omega$ satisfies the proper assumptions at $x^0$}. 

This result was established in a pioneering paper of S.~Zaremba \cite{Z10} for  the Laplace equation in a 3-dimensional domain $\Omega$ having interior touching ball at $x^0$
and generalized by G.~Giraud   \cite{G32}-\cite{G33} to equations with H{\"o}lder continuous leading coefficients and continuous lower order coefficients 
in domains $\Omega$ belonging to the class $C^{1, \alpha}$ with $\alpha \in (0,1)$. 

Notice that a related assertion about the negativity on $\partial\Omega$ of the normal derivative of the Green's function corresponding to the Dirichlet problem for the Laplace operator was proved much earlier for 2-dimensional smooth domains  by C.~Neumann in \cite{Neu88} (see also \cite{Ko01}). The result of \cite{Neu88} was extended for operators with the lower order coefficients by L.~Lichtenstein \cite{Lich24}. 
The same version of the Boundary Point Principle for the Laplacian and 3-dimensional domains satisfying a more flexible interior paraboloid condition was obtained by M.V.~Keldysch and M.A.~Lavrentiev in \cite{KeLa37}.

A crucial step in studying the Boundary Point Principle was made by E.~Hopf \cite{H52} and O.A.~Oleinik \cite{O52}, who simultaneously and independently proved the statement for the general elliptic equations with bounded coefficients and domains satisfying an interior ball condition at $x^0$. 

Later the efforts of many mathematicians were focused on  generalization of the Boundary Point Principle in several directions (for the details we refer the reader to \cite{AZ11} and \cite{Alv11} and references therein). Among these directions are the extension of the class of operators and the class of solutions, as well as the weakening of assumptions on the boundary.

The widening of the class of operators to singular/degenerate ones was made in the papers \cite{KaHim75}, \cite{KaHim77} and  \cite{AZ11}, while the uniform elliptic operators with unbounded lower order coefficients were studied in \cite{S10} and \cite{Naz12} (see also \cite{NU09}). We mention also the publications \cite{T83} and \cite{MSh15} where the Boundary Point Principle was established for a class of degenerate quasilinear operators including the $p$-Laplacian.

We note that before 2010 all the results were formulated for classical solutions, i.e. $u\in C^2\left( \Omega\right) $. The class of solutions was expanded  in \cite{S10} to  strong generalized solutions with Sobolev's second order derivatives. The latter requirement seems to be natural in studying of nondivergence elliptic equations. 

The reduction of the  assumptions on the boundary of $\Omega$ up to $C^{1, \text{Dini}}$-regularity was realized for various elliptic operators in the papers \cite{W67}, \cite{Him70} and \cite{Lie85}  (see also \cite{S08}). A weakened form of the Hopf-Oleinik lemma (the existence of a boundary point $x^1$ in any neighborhood of $x^0$ and a direction $\ell$ such that $\frac{\partial u}{\partial \ell}(x^1) \neq 0$)  was proved in \cite{N83} for a much wider class of domains including all Lipschitz ones. We mention also the paper \cite{Sw97} where the  behavior  of superharmonic functions near the boundary of  2-dimensional domains with corners is described  in terms of the main eigenfunction of the Dirichlet Laplacian.

The sharpness of some requirements was confirmed by corresponding counterexamples constructed in \cite{W67}, \cite{Him70}, \cite{KaHim75}, \cite{S08}, \cite{AZ11} and \cite{Naz12}. In particular, the counterexamples from \cite{W67}, \cite{Him70} and \cite{S08} show that the Hopf-Oleinik result fails for domains lying entirely in non-Dini paraboloids. 

The main result of our paper is a new counterexample showing the sharpness of the Dini condition for the boundary of $\Omega$. The simplest version of this counterexample can be formulated as follows: \vspace{0.1cm}

\noindent
\textit{Let $\Omega$ be a convex domain in $\mathbb{R}^n$, let $\partial \Omega$ in a neighborhood of the origin be described by the equation $x_n=F\left( x'\right)$ with $F \geqslant 0$ and
$F(0)=0$,
and let $u \in W^2_{n,\, loc}\left( \Omega\right) \cap C\left( \overline{\Omega}\right) $ be a solution of the uniformly elliptic equation}
$$
-a^{ij}(x)D_iD_ju=0 \quad \textit{in}\quad \Omega.
$$
\textit{Suppose also that $u\big|_{\partial\Omega}$  vanishes at a neighborhood of the origin.} 
\textit{If, in addition, the function $\delta (r)=\sup\limits_{|x'|\leqslant r}\frac{F(x')}{|x'|}$ is not Dini continuous at zero, then $\frac{\partial u}{\partial \mathbf{n}}(0)=0$. } \vspace{0.1cm}

Thus, it turns out that for convex domains the Dini continuity assumption on $\delta (r)$ is necessary and sufficient for the validity of the Boundary Point Principle. We emphasize that in our counterexample the Dini condition fails for  supremum of $F(x') /|x'|$, while in all the previous results of this kind it fails for infimum of $F(x') /|x'|$. In other words, we show that the violating of the Dini condition just in one direction  causes  the lack of the Hopf-Oleinik lemma.






\subsection{Notation and Conventions}

Throughout the paper we use the following notation:

\noindent$x=(x',x_n)=(x_1,\dots,x_{n-1},x_n)$ is a point in
${\mathbb R}^n$;

\noindent
$
\mathbb{R}^n_+=\left\lbrace x\in \mathbb{R}^n : x_n >0\right\rbrace; 
$

\noindent $|x|, |x'|$ are the Euclidean norms in the corresponding
spaces;

\noindent
$\chi_{\mathcal{E}}$ denotes the characteristic function of the set $\mathcal{E} \subset \mathbb{R}^n$;



\noindent $\Omega$ is a bounded domain in ${\mathbb R}^n$
with boundary $\partial\Omega$;






\noindent $\mathcal{P}_{r,h}(\overline{x}')=\left\lbrace x\in \mathbb{R}^n: |x'-\overline{x}'|<r, 0<x_n<h\right\rbrace $; 
\quad $\mathcal{P}_{r}(\overline{x}')=\mathcal{P}_{r,r}(\overline{x}')$;

\noindent $\mathcal{P}_{r,h}=\mathcal{P}_{r,h}(0)$; \quad $\mathcal{P}_r=\mathcal{P}_r(0)$;



\noindent $B_r(x^0)$ is the open ball in $\mathbb{R}^n$ with center
$x^0$ and radius $r$; \quad $B_r=B_r(0)$;



\medskip

\noindent For $r_1<r_2$ we define the annulus
$\mathcal{B}(x^0,r_1,r_2)=B_{r_2}(x^0)\setminus \overline{B_{r_1}(x^0)}$.



\noindent $v_+=\max\left\lbrace v,0\right\rbrace$ ,\quad $v_-=\max\left\lbrace -v,0\right\rbrace $.

\noindent $\|\cdot\|_{\infty, \Omega}$ denotes the norm in $L_{\infty}\left( \Omega \right)$. \vspace{0.2cm} 

\noindent We adopt the convention that the indices $i$ and $j$ run
from $1$ to $n$. We also adopt the convention regarding summation
with respect to repeated indices.

\medskip
\noindent $D_i$ denotes the operator of (weak) differentiation with respect
to  $x_i$;

\noindent
$D=\left( D',D_n\right) =\left( D_1, \dots,D_{n-1},D_n\right) $.


\medskip
\noindent ${\cal L}$ is a linear uniformly elliptic operator with
measurable coefficients:
\begin{equation} \label{numer}
{\cal L}u\equiv -a^{ij}(x)D_iD_ju+b^i(x)D_iu,\quad\qquad
\nu {\cal I}_n\le (a^{ij}(x))\le\nu^{-1}{\cal I}_n,
\end{equation} 
where ${\cal
I}_n$ is identity $(n\times n)$-matrix. We denote $\mathbf{b}(x)=\left( b^1(x), \dots, b^n(x)\right) $.


\medskip
\noindent We use letters $C$ and $N$ (with or without indices) to
denote various constants. To indicate that, say, $C$ depends on some
parameters, we list them in the parentheses: $C(\dots)$.

\begin{definition}
We say that a function $\sigma : [0,1]\rightarrow \mathbb{R}_+$ belongs to the class~$\mathcal{D}_1$~if 
\begin{itemize}
\item $\sigma$ is increasing, and $\sigma (0)=0$, and $\sigma (1)=1$;
\item $\sigma (t)/t$ is summable and decreasing.
\end{itemize}
\end{definition} \vspace{0.1cm}

\begin{remark}
Our assumption about decay of $\sigma (t)/t$  is not restrictive. Indeed, for any increasing function $\sigma :[0,1]\rightarrow \mathbb{R}_+$ satisfying $\sigma (0)=0$ and $\sigma (1)=1$ and having summable $\sigma (t)/t$, we can define
$$
\widetilde{\sigma} (t)=t \sup\limits_{\tau \in [t,1]}\frac{\sigma (\tau)}{\tau}, \qquad t\in (0,1).
$$
It is easy to see that $\widetilde{\sigma} \in \mathcal{D}_1$, $\widetilde{\sigma}(t)/t$ decreases and
$
\sigma (t) \leqslant \widetilde{\sigma} (t), \quad \forall t\in (0,1].
$ 
\end{remark} \vspace{0.1cm}

\begin{definition}
Let a function $\sigma$ belong to the class ~$\mathcal{D}_1$. We define the function $\mathcal{J}_{\sigma}$ as follows
\begin{equation} \label{J-sigma}
\mathcal{J}_{\sigma}(s):=\int\limits_{0}^{s}\frac{\sigma (\tau)}{\tau}d \tau.
\end{equation}
\end{definition}

\begin{remark} \label{zamechanie}
Decreasing of  $\sigma (t) /t$ implies
\begin{equation} \label{relation-1}
\sigma (t) \leqslant \mathcal{J}_{\sigma} (t) \qquad \forall t\in [0,1].
\end{equation}
In addition, for $t \leqslant t_0\leqslant 1$ we have
\begin{equation} \label{relation-sigma}
\sigma \left( t / t_0\right) =\frac{\sigma \left( t / t_0\right) }{t /t_0} \cdot t /t_0 \leqslant \frac{\sigma (t)}{t} \cdot t/t_0= \frac{\sigma (t)}{t_0},
\end{equation}
and, similarly,
\begin{equation} \label{relation-J-sigma}
\mathcal{J}_{\sigma}\left( t / t_0\right) \leqslant \frac{\mathcal{J}_{\sigma}(t)}{t_0}.
\end{equation} 
\end{remark} \vspace{0.1cm}

\begin{definition}
We say that a function $\zeta $ satisfies the Dini condition at zero if
$$
|\zeta (r)| \leqslant C \sigma (r),
$$
and $\sigma$ belongs to the class $\mathcal{D}_1$.
\end{definition}


\section{Preliminaries}
\subsection{Properties of $\Omega$}
Let $\Omega$ be a bounded domain in $\mathbb{R}^n$. Without loss of generality we may assume   $0 \in \partial\Omega$.

Suppose that $\Omega$ is locally convex in a neighborhood of the origin. Without restriction the latter means that for some $0<\mathcal{R}_0 \leqslant 1$ we have 
$$
\mathcal{P}_{\mathcal{R}_0}\cap \Omega =\left\lbrace (x',x_n) \in \mathbb{R}^n: |x'|\leqslant \mathcal{R}_0,  F(x')< x_n <\mathcal{R}_0\right\rbrace ,
$$
where $F$ is a convex nonnegative function satisfying $F(0)=0$.

For $r \in (0,\mathcal{R}_0)$ we define the functions $\delta=\delta (r)$ and $\delta_1~=~\delta_1(r)$ by the formulas
\begin{equation}
\delta (r):=\max\limits_{|x'|\leqslant r}\frac{F(x')}{|x'|}, \qquad \delta_1(r):=\max\limits_{|x'|\leqslant r} |\nabla F(x')|.
\label{delta-definition}
\end{equation}

\begin{lemma}\label{Proposition}
The following statements hold:
\begin{itemize}
\item[(a)] $\delta_1(r) \to 0$\ as\ $r \to 0$\ iff $\delta (r) \to 0$ as $r\to 0$.
\item[(b)]  $\delta_1(r)$ satisfies the Dini condition at zero iff $\delta (r)$ satisfies the Dini condition at zero.
\end{itemize}
\end{lemma}

\begin{proof}
By convexity of $F$, we have for any $x'$ and $z'$ the estimate
\begin{equation}
F(z')\geqslant F(x')+\nabla F(x')\cdot (z'-x').
\label{**}
\end{equation}
Therefore,
$$
|\nabla F(x')|\geqslant \nabla F(x') \cdot \frac{x'}{|x'|}\geqslant \frac{F(x')}{|x'|},
$$
and, consequently,
\begin{equation}
\delta_1(r) \geqslant \delta (r).
\label{***}
\end{equation}

On the other hand, for any $r<\mathcal{R}_0/2$ we can find a point $x'_{*}$ such that
$$
|\nabla F(x'_{*})|=\delta_1(r).
$$
Chosing 
$
z'=x'_{*}+r\dfrac{\nabla F(x'_{*})}{|\nabla F(x'_{*})|},
$
we easily deduce from (\ref{**}) the inequalities 
$$
 |z'| \leqslant 2r \quad \text{and} \quad F(z')\geqslant r\delta_1(r),
$$
which provide
\begin{equation}
\delta (2r)\geqslant \delta (|z'|)\geqslant \dfrac{\delta_1(r)}{2}.
\label{****}
\end{equation} 

Combining (\ref{***}) and (\ref{****}) we conclude that statement {\it{(a)}} is obvious and the integrals
$$
\int\limits_0^{\mathcal{R}_0}\frac{\delta(r)}{r}dr \quad \text{and}\quad \int\limits_{0}^{\mathcal{R}_0}\frac{\delta_1(r)}{r}dr
$$
converge simultaneously.
\end{proof}

If $\delta (r)$ does not converge to zero as $r \rightarrow 0$, we can easily see that  the domain $\Omega$ is contained in a dihedral wedge with the angle less than $\pi$ and the edge going through the origin. For this case the statement of  \textbf{Main Theorem} is proved already in \cite[Theorem~4.3]{AN01}. By this reason we will assume throughout  this paper that
\begin{equation} \label{delta-to-zero}
\delta (r)\rightarrow 0 \quad \text{as}\quad r\rightarrow 0.
\end{equation}

In view of (\ref{delta-to-zero}), it is evident that $\delta$ and $\delta_1$ are moduli of continuity at the origin of the functions $F(x')/|x'|$ and $|\nabla F(x')|$, respectively. \vspace{0.2cm}



\subsection{Properties of ${\cal X} \left( \Omega\right) $}
 Let ${\cal X}(\Omega)$ be a function space with the norm $\|\cdot\|_{{\cal{X}},\Omega}$. For $\Omega_1 \subset \Omega$ we will assume that
$$
\|f\|_{\mathcal{X}, \Omega_1}=\|f \cdot \chi_{\Omega_1}\|_{\mathcal{X}, \Omega}.
$$

\vspace{0.2cm}

We suppose that ${\cal X}(\Omega)$ has the following properties: 

\begin{itemize}
\item[(i)] 
\textit{For arbitrary measurable function} $g$ \textit{defined in} $\Omega$
\textit{and any function}
 
$f \in \mathcal{X}\left( \Omega\right)$
\textit{ the inequality}  
$|g(x)| \leqslant |f(x)|$ \textit{implies} $g \in \mathcal{X}\left( \Omega\right)$ \textit{and} 

$\|g\|_{\mathcal{X}, \Omega} \leqslant \|f\|_{\mathcal{X}, \Omega}$;
\item[(ii)] 
\textit{For}  $f_k \in \mathcal{X}\left( \Omega\right) \ \textit{the convergence}\ f_k \searrow 0$ \textit{a.e.\, in} $\Omega$ \textit{implies}
 
$\|f_k\|_{\mathcal{X}, \Omega}\rightarrow 0$.
\end{itemize}

Using the terminology of classic monograph of Kantorovich and Akilov \cite{KA82} we may say that $\mathcal{X}\left( \Omega\right) $ is the ideal functional space with order continuous monotone norm (see \cite[\S3, Chapter IV, Part I]{KA82} for more details). 

\vspace{1ex}
We will also assume that 
\begin{itemize}
\item[(iii)]  $\mathcal{X}_{loc}\left( \Omega\right) \ \textit{contains the Orlicz space}\ L_{\Phi, loc}\left( \Omega\right) \ \textit{with}\ \Phi(\varsigma)=e^{\varsigma}-\varsigma-1.$
\end{itemize}

\vspace{0.2cm}
Finally, the basic assumption about $\mathcal{X} \left( \Omega\right)$ is the Aleksandrov-type maximum principle.
Namely, we denote by $\mathcal{W}^2_{\mathcal{X}, loc}\left( \Omega\right) $ the set of the functions $u$ satisfying $D\left( Du\right) \in \mathcal{X}_{loc}\left( \Omega \right) $, and suppose that if $u\in \mathcal{W}^{2}_{\mathcal{X}, loc}\left( \Omega \right) \cap \mathcal{C}\left( \overline{\Omega}\right)  $,
$u|_{\partial\Omega}\le0$, and $|\textbf{b}|\in \mathcal{X}\left( \Omega\right)$  then 
\begin{equation}
u\leqslant N_0(n,\nu, \|\textbf{b}\|_{\mathcal{X}, \Omega}
)\cdot \textup{diam}(\Omega)
\cdot\|({\cal L}u)_+\|_{{\cal X}, \left\lbrace u>0\right\rbrace }.
\label{max-principle}
\end{equation} 

\begin{remark}
It is well known from \cite{Al60},  \cite{B61} and \cite{Al63} (see also survey \cite{Naz07} for further references) that $L_n(\Omega)$ has  property  (\ref{max-principle}). It is also evident that properties (i)-(iii) are satisfied in $L_n\left( \Omega\right) $. Therefore, $L_n(\Omega)$ can be treated as a "basic" example of ${\cal X}(\Omega)$. As other examples of the space ${\cal{X}}(\Omega)$ we mention some Lebesgue weighted spaces with power weights (see \cite{Naz01}).
\end{remark}

\begin{remark}
Unlike the natural properties (i)-(ii),  assumption (iii) is rather "technical" one. Without (iii), our arguments from the proof of Step~3 in Theorem~\ref{osc-estimate} are not applicable to the approximating operator $\mathcal{L}_{\varepsilon}$. So, we can not withdraw (iii) in abstract setting. However, in all known examples of $\mathcal{X}\left( \Omega \right)$  the property (iii) is satisfied.
\end{remark}

\begin{remark} \label{X-equal-Ln}
Some of the statements, that will be referred to in the sequel, were proved earlier just for the case ${\cal X}(\Omega)=L_n(\Omega)$. However, if all the arguments are based only on the Aleksandrov-type maximum principle, these statements remain valid for an arbitrary considered space ${\cal X}(\Omega)$. In such cases, we will refer without any further explanation.
\end{remark} 

\vspace{0.2cm}

We also need the following convergence lemmas.

\begin{lemma}\label{analog-Lebesgue-theorem}
Let $\left\lbrace f_j \right\rbrace $ be a sequence of measurable functions on $\Omega$, and let $f \in \mathcal{X}\left( \Omega \right) $. Suppose also that $f_j\rightarrow 0$  in measure on $\Omega$, and $|f_j(x)| \leqslant |f(x)|$.

Then 
\begin{equation} \label{analog-Lebesgue}
\|f_j\|_{\mathcal{X}, \Omega}\rightarrow 0 \quad \textit{as} \quad j \rightarrow \infty.
\end{equation}
\end{lemma}

\begin{proof} We argue by a contradiction. Suppose (\ref{analog-Lebesgue}) fails. Then there exists a subsequence $\left\lbrace f_{j_k}\right\rbrace $ satisfying
\begin{equation} \label{contradiction}
\|f_{j_k}\|_{\mathcal{X}, \Omega} \geqslant \varepsilon >0, \qquad \forall k \in \mathbb{N}.
\end{equation}
Due to the Riesz theorem, there exists also a sub-subsequence $\left\lbrace f_{j_{k_l}}\right\rbrace $ such that
$$
f_{j_{k_l}}\rightarrow 0\quad \text{a.e.\  in} \quad \Omega.
$$
For simplicity of notation we renumber  the latter subsequence $\left\lbrace f_{j_{k_l}}\right\rbrace $ and denote its elements again by $f_j$. \vspace{0.2cm}

Setting
$
\tilde{f}_k:=\sup\limits_{j \geqslant k} |f_j|
$
we can easily see that $\tilde{f}_k\searrow 0$ a.e. in $\Omega$. Now, taking into account properties (i) and (ii) of the space $\mathcal{X}\left( \Omega\right) $ we immediately get a contradiction with inequalities (\ref{contradiction}). The proof is complete.
\end{proof}
\vspace{0.2cm}

\begin{lemma}\label{merzost'-to-0}
Let $f\in \mathcal{X}\left( \Omega\right) $, and let $\mu (\rho):=\sup\limits_{x \in \Omega}\|f\|_{\mathcal{X}, B_{\rho} (x)\cap \Omega}$.

Then
$$
\mu (\rho) \rightarrow 0 \quad as\ \rho \rightarrow 0.
$$
\end{lemma}

\begin{proof} For every $\rho >0$ there exists a point $x^*=x^*(\rho)\in \Omega$ such that
$$
\|f\|_{\mathcal{X}, B_{\rho}(x^*)\cap\Omega} \geqslant \frac{1}{2}\mu (\rho).
$$

Next, for the sequence
$
f_{\rho}:=f \cdot \chi_{B_{\rho }(x^*)}
$
it is evident that $|f_{\rho}|\rightarrow 0$  in measure on $\Omega$. 
Application of Lemma~\ref{analog-Lebesgue-theorem} finishes the proof.
\end{proof} 

\begin{remark}
We call $\mu (\rho):=\sup\limits_{x\in \Omega}\|f\|_{\mathcal{X}, B_{\rho}(x)\cap \Omega}$ the modulus of continuity of function $f$ in $\mathcal{X} \left( \Omega\right) $.
\end{remark} \vspace{0.2cm}

\begin{lemma} \label{operator-approximation}
Let 
$D(Du) \in \mathcal{X}\left( \Omega \right) $, let $\mathcal{L}$ be defined by (\ref{numer}), 
and  let 
$\mathcal{L}u~\in~\mathcal{X} \left( \Omega\right) $. 
There exist the
family of operators
$$
\mathcal{L}_{\varepsilon}=-a^{ij}_{\varepsilon}(x)D_iD_j+b^{i}_{\varepsilon}(x)D_i
$$
with smooth coefficients $a^{ij}_{\varepsilon}$ and bounded coefficients $b^i_{\varepsilon}$ satisfying
\begin{gather}
\nu {\cal I}_n\le (a^{ij}_{\varepsilon}(x))\le\nu^{-1}{\cal I}_n, \qquad x\in \Omega, \label{app-ellipticity-a}\\
|b^i_{\varepsilon}(x)| \leqslant |b^i(x)|, \qquad  x\in\Omega, \label{app-estimate-b}\\
\|\left( \mathcal{L}-\mathcal{L}_{\varepsilon}\right) u\|_{\mathcal{X}, \Omega}\rightarrow 0\quad \text{as}\quad \varepsilon\rightarrow 0, \label{app-convergence-b}
\end{gather}
respectively.
\end{lemma}
\begin{proof} We start with extension of $a^{ij}$ on the whole $\mathbb{R}^n$ by the identity matrix and denote by $a^{ij}_{\varepsilon}$  the standard mollification of  extended functions $a^{ij}$.
By construction, the coefficients $a^{ij}_{\varepsilon}$ are smooth functions converging as $\varepsilon \rightarrow 0$ to $a^{ij}$ a.e. in $\Omega$. Moreover, it is clear that inequalities (\ref{app-ellipticity-a}) are true. 

Further, we set 
\begin{equation} \label{def-b^i_e}
\widetilde{b}^{i}_{\varepsilon} (x):=\min \left\lbrace |b^{i}(x)|, \varepsilon^{-1}\right\rbrace \cdot \text{sign}\, b^{i}(x).
\end{equation}
In view of (\ref{def-b^i_e}), it is evident that $\widetilde{b}^i_{\varepsilon}D_iu$ converges as $\varepsilon\rightarrow 0$ to $b^iD_iu$  almost everywhere~in~$\Omega$. We claim that it is possible to change $\widetilde{b}^i_{\varepsilon}$ such that the "corrected coefficients" $b^i_{\varepsilon}$  satisfy
\begin{equation} \label{for-majoranta}
|b^i_{\varepsilon}D_iu| \leqslant |b^iD_iu| \quad \text{in} \quad \Omega.
\end{equation}
Indeed, if $|\widetilde{b}^i_{\varepsilon}D_iu|\leqslant |b^iD_iu|$ in $\Omega$ then (\ref{for-majoranta}) holds with $b^i_{\varepsilon}\equiv \widetilde{b}^i_{\varepsilon}$. Otherwise,
consider a point $x^0 \in \Omega$ where $|\widetilde{b}^{i}_{\varepsilon}(x^0)D_iu(x^0)|>|b^i(x^0)D_iu(x^0)|$. 
\begin{itemize}
\item[$\boxed{a)}$] Let $\widetilde{b}^i_{\varepsilon}(x^0)D_iu(x^0) > b^i(x^0)D_i u(x^0) \geqslant 0$. In this case we decrease all the coefficients $\widetilde{b}^i_{\varepsilon}(x^0)$ corresponding to the positive summands such that the both sums $b^i_{\varepsilon}D_iu$ and $b^iD_iu$  becomes equal.
\item[$\boxed{b)}$] Let $\widetilde{b}^i_{\varepsilon}(x^0)D_iu(x^0) < b^i(x^0)D_i u(x^0) \leqslant 0$. In this case we decrease all the coefficients $\widetilde{b}^i_{\varepsilon}(x^0)$ corresponding to the negative summands such that the both sums $b^i_{\varepsilon}D_iu$ and $b^iD_iu$  becomes equal.
\item[$\boxed{c)}$] Finally, let $\widetilde{b}^i_{\varepsilon}(x^0)D_iu(x^0)$ and $b^i(x^0)D_iu(x^0)$ have different signs. In this case we apply to $-b^i_{\varepsilon}(x^0)$ the arguments from  case a) or from case~b), respectively.
\end{itemize}
Due to construction, the "corrected sum" $b^i_{\varepsilon}D_iu$ also converges as $\varepsilon \to 0$ to $b^iD_iu$ a.e. in $\Omega$, and pointwise inequalities (\ref{app-estimate-b}) hold true. 

Finally, taking into account (\ref{for-majoranta}) and  applying Lemma~\ref{analog-Lebesgue-theorem} we get                (\ref{app-convergence-b}).
\end{proof}



\section{Gradient estimates near the boundary}

\begin{lemma} \label{analog-lemma1-AU}
Let ${\cal{N}} \subset\mathbb{R}^n_+ $ be an open set, let $\gamma=\frac{\nu}{\sqrt{n-1}}$, let $\rho >0$, and let
$$
\Pi_{\rho}=\left\{y \in \mathbb{R}^n : |y_i|<\rho \quad \text{\it for} \quad i=1,\dots,n-1; \quad
0<y_n<\gamma\rho\right\}.
$$
We assume that $|\mathbf{b}|\in \mathcal{X}\left( \mathcal{N} \right) $ and a function $v$ satisfies the conditions
$$
v \in \mathcal{W}^2_{\mathcal{X}, loc} \left( \mathcal{N} \right), \quad
v \geqslant 0 \quad \text{\it in}\quad
\Pi_{\rho}, \quad v \geqslant k=\textit{const}>0 \quad \text{\it on}\quad \partial
{\cal N} \cap \overline{\Pi}_{\rho}.
$$

Then 
$$
v \geqslant  C_1 k-C_2 k \|\mathbf{b}\|_{\mathcal{X}, \, \mathcal{N} \cap \Pi_{\rho}}-C_3\rho \|(\mathcal{L}v)_-\|_{\mathcal{X}, \, \mathcal{N} \cap \Pi_{\rho}} \quad \text{\it in}
\quad {\cal N}\cap 
B_{\frac{\gamma\rho}{4}}(z)  ,
$$
where 
$z=(0,\dots,0,\frac{1}{2}\gamma\rho)$, while 
$C_1=\frac{1}{16}\left( 1-\gamma^2\right) $,  $C_2=C_2(n, \nu, \|\mathbf{b}\|_{\mathcal{X}, \, \mathcal{N}})$, and 
$C_3~=~C_3(n, \nu, \|\mathbf{b}\|_{\mathcal{X}, \, \mathcal{N}})$.
\end{lemma}

\vspace{0.2cm}
\begin{proof} The proof is similar in spirit to \cite[Lemma 1]{AU95}.

Consider the barrier function
$$
\psi (y)=k \left[\left( 1-\frac{y_n}{\gamma \rho}\right)^2 -\frac{|y'|^2}{\rho^2}\right].
$$
An elementary computation gives
$$
{\cal L} \psi \leqslant k \left(\frac{2(n-1)}{\rho^2}
\nu^{-1}-\frac{2}{\gamma^2 \rho^2}\nu\right)+|\mathbf{b}||D\psi| \leqslant N_1(n,\nu)|\mathbf{b}|\frac{k}{\rho}  \quad \text{in}
\quad \Pi_{\rho}.
$$
Moreover, setting
$$
\begin{aligned}
\mathcal{S}_1&=\{y \in \partial ({\cal N} \cap \Pi_{\rho}): |y_i|=\rho \quad \text{for some} \quad i=1,\dots,n-1\},\\
\mathcal{S}_2&=\{y \in \partial ({\cal N} \cap \Pi_{\rho}):
y_n=\gamma\rho \}
\end{aligned}
$$
we have
\begin{align*}
\psi \big|_{\mathcal{S}_1\cup \mathcal{S}_2} &\leqslant 0 \leqslant v,\\
\psi\big|_{\partial {\cal N}\cap \overline{\Pi}_{\rho}}&\leqslant k\leqslant
v\big|_{\partial {\cal N}\cap \overline{\Pi}_{\rho}}.
\end{align*}

Applying inequality (\ref{max-principle}) in ${\cal N}\cap \Pi_{\rho}$ to the
difference $\psi-v$ we obtain
$$
\psi-v \leqslant N_0 \cdot \textup{diam}(\Pi_{\rho})\cdot
\|(\mathcal{L}\psi-\mathcal{L}v)_+\|_{\mathcal{X}, \, \mathcal{N} \cap \Pi_{\rho}} \quad \text{in} \quad {\cal N}\cap \Pi_{\rho},
$$
and, consequently,
$$
\begin{aligned}
v &\geqslant k \left[\left( 1-\frac{\frac{3}{4}\gamma \rho}{\gamma \rho}\right)^2 -\frac{\gamma^2 \rho^2}{16 \rho^2}\right] -C_2 k \|\mathbf{b}\|_{\mathcal{X}, \, \mathcal{N} \cap \Pi_{\rho}}-C_3\rho \|(\mathcal{L}v)_-\|_{\mathcal{X}, \, \mathcal{N} \cap \Pi_{\rho}}\\
&=\frac{(1-\gamma^2)}{16}k-C_2 k \|\mathbf{b}\|_{\mathcal{X}, \, \mathcal{N} \cap \Pi_{\rho}}-C_3\rho \|(\mathcal{L}v)_-\|_{\mathcal{X}, \, \mathcal{N} \cap \Pi_{\rho}} \quad \text{in} \quad {\cal N} \cap B_{\frac{\gamma\rho }{4}}(z).
\end{aligned}
$$
\end{proof}

\vspace{0.3cm} 
Our next statement is a version of Theorem~2.3 \cite{Naz12}.

\begin{lemma}\label{lemma2.3-LU88}
Let $v \in \mathcal{W}^2_{\mathcal{X}, loc} \left( \Omega \right) \cap \mathcal{C}\left( \overline{\Omega}\right) $, let $v\big|_{\partial\Omega}=0$, and let $|\mathbf{b}|\in \mathcal{X} \left( \Omega \right)$.
Suppose also that 
for all $\rho \leqslant \rho_{*} \leqslant 1$ the inequalities
$$
\|b^n\|_{\mathcal{X}, \mathcal{P}_{\rho} \cap \Omega} \leqslant \mathfrak{B}\sigma \left( \rho / \rho_{*}\right), \quad
\|\left( \mathcal{L}v\right)_{+} \|_{\mathcal{X}, \mathcal{P}_{\rho} \cap \Omega} \leqslant \mathfrak{F}\sigma \left( \rho / \rho_{*}\right)
$$
hold true. Here $\mathfrak{B}$ and $\mathfrak{F}$ are some positive constants, while a function $\sigma$ belongs to $\mathcal{D}_1$.

Then
\begin{equation} \label{2.3-LU88}
\sup\limits_{0<x_n<\rho}\frac{v(0,x_n)}{x_n}\leqslant C_4 \left( \rho^{-1}\sup\limits_{\mathcal{P}_{\rho}\cap \Omega}v +\mathfrak{F} \mathcal{J}_{\sigma} \left( \rho /\rho_{*}\right)\right), \quad \forall \rho\leqslant \rho_{*}.
\end{equation}
Here the constant $C_4$ depends on $n$, $\nu$,  $\mathfrak{B}$, $\sigma$, and on the moduli of continuity of $|\mathbf{b}'|$ in $\mathcal{X}\left( \mathcal{P}_{\rho_{*}}\cap\Omega\right)$,  whereas $\mathcal{J}_{\sigma}$ is a function defined by formula (\ref{J-sigma}).
\end{lemma}

\begin{remark}
We recall that $0 \in \partial\Omega$.
\end{remark}
\begin{proof} First, we assume that $\rho \leqslant \overline{\rho}$, where $\overline{\rho} \leqslant \rho_{*}$ will be fixed later.
Following \cite{Naz12} we introduce the sequence of cylinders $\mathcal{P}_{\rho_k, h_k}$, $k\geqslant 0$, where $\rho_k=2^{-k}\rho$, $h_k=\zeta_k\rho_k$, while the sequence $\zeta_k\downarrow 0$ will be chosen later.

We set $w_k=v-M_kx_n$, where the quantities $M_k$, $k \geqslant 1$ are defined as
$$
M_k=\sup\limits_{\mathcal{P}_{\rho_k, h_{k-1}}\cap \Omega}\frac{v(x)}{\max{\left\lbrace x_n, h_k\right\rbrace }} \geqslant 
\sup\limits_{\left\lbrace \mathcal{P}_{\rho_k, h_{k-1}} \setminus \mathcal{P}_{\rho_k, h_k}\right\rbrace \cap \Omega}\frac{v(x)}{x_n}.
$$
It is easy to see that $w_k \leqslant 0$ on $\partial\Omega \cap \overline{\mathcal{P}}_{\rho_k,h_k}$, while the definition of $M_k$ gives $w_k \leqslant 0$ on the top of the cylinder $\mathcal{P}_{\rho_k, h_k}$.

Let $x^0 \in \mathcal{P}_{\rho_k-h_k, h_k}\cap \Omega$. Taking into account Remark~\ref{X-equal-Ln} we apply the so-called "boundary growth lemma" (see, for instance,  \cite[Lemma~2.5']{LU85} or  \cite[Lemma~2.6]{S10} or  \cite[Lemma~2.2]{Naz12}) to the (positive) function $M_kh_k~ -~w_k$ in 
$\mathcal{P}_{h_k}({x^0}\vphantom{x}')\cap \Omega$. It gives for $x\in \mathcal{P}_{h_k/2, h_k}\left({x^0}\vphantom{x}'\right) \cap \Omega$
\begin{equation}
\begin{aligned}
M_kh_k-w_k(x) \geqslant &M_kh_k \left[ \vartheta -N_2\|\textbf{b}\|_{\mathcal{X}, \mathcal{P}_{\rho_k}\cap \Omega}
\right] \\
&-N_3h_k
\|\left( \mathcal{L}w_k\right) _+\|_{\mathcal{X}, \mathcal{P}_{h_k}\left({x^0}\vphantom{x}'\right)\cap \Omega},
\end{aligned}
\label{Naz12-10}
\end{equation} 
where $\vartheta=\vartheta (n, \nu, \sigma, \mathfrak{B}) \in (0,1)$, the positive constant $N_2$ depends on the same parameters as $\vartheta$ whereas the positive constant $N_3$ is completely defined by the values of $n$, $\nu$ and $\mathfrak{B}$. We suppose that $\overline{\rho}$ is so small that the quantity in the square brackets is greater than $\vartheta /2$. Further, direct calculation shows that the assumptions of our lemma imply
\begin{align*}
\|\left( \mathcal{L}w_k\right) _+\|_{\mathcal{X}, \mathcal{P}_{h_k}\left({x^0}\vphantom{x}'\right)\cap \Omega} &\leqslant 
\|\left( \mathcal{L}v\right) _+\|_{\mathcal{X}, \mathcal{P}_{h_k}\left({x^0}\vphantom{x}'\right)\cap \Omega}+
M_k\|b^n\|_{\mathcal{X}, \mathcal{P}_{h_k}\left({x^0}\vphantom{x}'\right)\cap \Omega}\\
&\leqslant
\left( \mathfrak{F} +M_k \mathfrak{B}\right) \sigma ( \rho_k /\rho_{*}) .
\end{align*}
Substituting the last inequality into (\ref{Naz12-10}) and taking 
supremum w.r.t. $x^0$ we obtain
$$
\sup\limits_{\mathcal{P}_{\rho_k-h_k,h_k}\cap \Omega} w_k \leqslant M_kh_k \left[  1- \vartheta/2+N_2 \mathfrak{B} \sigma ( \rho_k /\rho_{*})\right]   +N_3 h_k\mathfrak{F} \sigma ( \rho_k /\rho_{*}) .
$$

Repeating previous arguments provides for integer $m \leqslant \frac{\rho_k}{h_k}$ the inequality
$$
\sup\limits_{\mathcal{P}_{\rho_k-mh_k, h_k}\cap \Omega} w_k \leqslant M_kh_k\left[ (1-\vartheta /2)^m+N_2 \mathfrak{B}\,\frac{\sigma ( \rho_k/\rho_{*}) }{\vartheta /2}\right] + N_3 h_k \mathfrak{F} \,\frac{\sigma ( \rho_k/\rho_{*} ) }{\vartheta/2}.
$$
Setting $m=\lfloor \frac{\rho_{k+1}}{h_k}\rfloor$, we arrive at
\begin{align*}
\sup\limits_{\mathcal{P}_{\rho_{k+1}, h_k} \cap \Omega}w_k &\leqslant \frac{M_k h_k}{1-\vartheta/2}
\left( \exp{\left( -\lambda
\frac{\rho_{k+1}}{h_k}\right) } 
+N_2 \mathfrak{B} \, \frac{\sigma ( \rho_{k}/\rho_{*} ) }{\vartheta/2} \right)\\ 
&+N_3h_k\mathfrak{F}\,\frac{\sigma \left( \rho_{k}/\rho_{*}\right)}{(1-\vartheta/2) \vartheta/2},
\end{align*}
where $\lambda=-\ln{\left( 1-\vartheta/2\right) }>0$.
 
Therefore, for $x\in \mathcal{P}_{\rho_{k+1}, h_k}\cap \Omega$
\begin{equation}
\frac{w_k(x)}{\max{\left\lbrace x_n, h_{k+1}\right\rbrace } }\leqslant M_k \gamma_k + N_3\mathfrak{F}\,\frac{\sigma ( \rho_{k}/\rho_{*})}{(1-\vartheta/2) \vartheta/2} \cdot \frac{2\zeta_k}{\zeta_{k+1}},
\label{Naz12-11}
\end{equation}
where $\gamma_k=\frac{1}{1-\vartheta/2}\frac{2\zeta_k}{\zeta_{k+1}}\cdot \left( \exp{\left( -\frac{\lambda}{2\zeta_k}\right)} +
N_2 \mathfrak{B} \, \frac{\sigma ( \rho_{k}/\rho_{*} )}{\vartheta/2}\right) $.
\vspace{0.2cm}

Estimate (\ref{Naz12-11}) implies
\begin{align*}
M_{k+1}&\leqslant M_k\left( 1+\gamma_k\right) + N_3\mathfrak{F}\, \frac{\sigma ( \rho_{k}/\rho_{*})}{ (1-\vartheta/2)\vartheta/2} \cdot \frac{2\zeta_k}{\zeta_{k+1}} \\
&\leqslant M_1\cdot \prod\limits_{j=1}^k \left( 1+\gamma_j\right)+2N_3 \mathfrak{F} \cdot
\sum\limits_{j=1}^k\sigma ( \rho_{j}/\rho_{*}) \frac{\zeta_j}{\zeta_{j+1}} \cdot
\prod\limits_{l=j}^{k-1}(1+\gamma_l) .
\end{align*}
We set $\zeta_k=\frac{1}{k+k_0}$ and choose $k_0$ so large and $\overline{\rho}/\rho_{*}$ so small that $\gamma_1 \leqslant \frac{1}{2}$. Note that $k_0=k_0(n, \nu, \sigma, \mathfrak{B})$ while $\overline{\rho}/\rho_{*}$ depends  on the same parameters as $k_0$ and, in addition, on the moduli of continuity of $|\textbf{b}'|$ in $\mathcal{X}\left( \mathcal{P}_{\rho_{*}}\cap \Omega\right) $.

Now we observe that the first term in $\gamma_k$ forms a convergent series. The same is true for the second term, since
$$
\sum\limits_{k=1}^{\infty}\sigma ( 2^{-k}\rho/\rho_{*}) \asymp \int\limits_0^{\infty}
\sigma ( 2^{-s}\rho /\rho_{*}) ds \asymp \mathcal{J}_{\sigma} ( \rho /\rho_{*}) .
$$
Therefore, the infinite product $\Pi=\prod\limits_k
\left( 1+\gamma_k\right) $ also converges, and we obtain for $k>1$ the inequality
\begin{equation}
\begin{aligned}
M_k &\leqslant 
\Pi \cdot \left( M_1+2N_3\mathfrak{F}\cdot \sum\limits_{j=1}^k \sigma ( \rho_j / \rho_{*})\frac{\zeta_j}{\zeta_{j+1}}\right) \\
&\leqslant\Pi \cdot \left( M_1 +N_4(n, \nu, \sigma, \mathfrak{B})\mathfrak{F}\, \mathcal{J}_{\sigma}( \rho / \rho_{*}) \right). 
\end{aligned}
\label{Naz12-11a}
\end{equation}
Thus, all $M_k$ are bounded. 
It remains only to note that 
\begin{equation}
M_1 \leqslant \frac{1}{h_1}\sup\limits_{\mathcal{P}_{\rho/2}\cap \Omega}v.
\label{Naz12-11b}
\end{equation} 
Combining (\ref{Naz12-11a}) and (\ref{Naz12-11b}),
we arrive at 
\begin{equation} \label{strip-estimate-up-r/2}
\sup\limits_{0<x_n<\rho /2}\frac{v(0,x_n)}{x_n}\leqslant N_5(n, \nu, \sigma, \mathfrak{B}) \left( \rho^{-1}\sup\limits_{\mathcal{P}_{\rho /2} \cap \Omega}v+ \mathfrak{F}
\mathcal{J}_{\sigma}\left(\rho /\rho_{*}\right)\right).
\end{equation}

Further, it is easy to find a majorant for $\dfrac{v(0,x_n)}{x_n}$ for any $x_n \in [\rho /2,\rho )$ since
\begin{equation} \label{strip-estimate-after-r/2}
\sup\limits_{\rho /2 \leqslant x_n<\rho}\frac{v(0,x_n)}{x_n}\leqslant 2\rho^{-1}\sup\limits_{\rho /2 \leqslant x_n<\rho }v(0,x_n) \leqslant 
2\rho^{-1} \sup\limits_{\mathcal{P}_{\rho} \cap \Omega} v.
\end{equation}
Combination of (\ref{strip-estimate-up-r/2}) and (\ref{strip-estimate-after-r/2}) implies (\ref{2.3-LU88}) with $C_4=\max{\left\lbrace N_5, 2\right\rbrace }$ for $\rho \leqslant \overline{\rho}$. 

Now, we consider $\rho > \overline{\rho}$. If $x_n < \overline{\rho}$ then  the estimate
\begin{equation}
\frac{v(0, x_n)}{x_n}\leqslant 2N_5 \left( \overline{\rho}^{\,-1}\sup\limits_{\mathcal{P}_{\rho } \cap \Omega}v+ \mathfrak{F}
\mathcal{J}_{\sigma}\left(\rho /\rho_{*}\right)\right)
\label{Naz12-13}
\end{equation}
follows from the above arguments. Otherwise, i.e. for $x_n \geqslant \overline{\rho}$, inequality (\ref{Naz12-13}) is especially true. Thus, for $\rho > \overline{\rho}$ we again arrive at (\ref{2.3-LU88}) with 
$C_4=
 \max{\left\lbrace N_5, 2\right\rbrace } \overline{\rho}^{\,-1} 
$. The proof is complete.
\end{proof} \vspace{0.1cm}

\section{Main results}

Recall that $\Omega$ satisfies the assumptions from Subsection~2.1.
Throughout this section we shall suppose that $\mathcal{L}$ is defined by (\ref{numer}),  $|\mathbf{b}|\in \mathcal{X}\left( \Omega \right) $, and a function $u$ satisfies the following assumptions:
\begin{equation} \label{main-assumptions-on-u}
u\in \mathcal{W}^2_{\mathcal{X}, loc} \left( \Omega\right) \cap \mathcal{C}\left( \overline{\Omega}\right), \quad \mathcal{L}u=0 \quad \textit{in}\quad \Omega,  \quad u\big|_{\partial\Omega \cap \overline{\mathcal{P}}_{\mathcal{R}_0}}=0.
\end{equation}

\vspace{0.1cm}

\begin{thm} \label{osc-estimate}
Let the inequality
$$
\sup\limits_{x  \in \mathcal{P}_{\mathcal{R}_0/2}}\|b^n\|_{\mathcal{X}, \mathcal{P}_{\rho}(x')\cap \Omega} \leqslant \mathfrak{B} \sigma ( \rho /\mathcal{R}_0) 
$$
hold true  for all $\rho \leqslant \mathcal{R}\underline{\textbf{}}_0/2$. Here $\mathfrak{B}$ is a positive constant, and a function $\sigma \in \mathcal{D}_1$ satisfies
\begin{equation} \label{relation-2}
\mathcal{J}_{\sigma}(t)=o(\delta (t)) \quad \text{as}\quad t \to 0.
\end{equation}

 
Then, there exists  a sufficiently small positive number $R_0$ completely defined by $n$, $\nu$, $\mathcal{R}_0$, $\mathfrak{B}$, by the functions $\sigma$, $\delta$,
and by the moduli of continuity of $|\textbf{b}'|$ in $\mathcal{X}\left( \Omega\right) $ such that for any $r\in (0, R_0/2)$ we have
\begin{equation}
\osc{\Omega\cap \mathcal{P}_{r/4}}\frac {u(x)}{x_n}\leqslant \left( 1-\varkappa \delta (r)\right) \osc{\Omega\cap \mathcal{P}_{2r}}\frac {u(x)}{x_n}.
\label{1}
\end{equation}
Here
 the  constant $\varkappa \in (0;1)$ 
is completely determined by
$n$, $\nu$.
\end{thm}

\begin{proof}  The proof will be divided into 3 steps.
\paragraph{$\boxed{1.}$}
Our arguments are adapted
from \cite[Lemma~2]{AU95} and \cite[Lemma~3]{U96}. Let us denote
$$
m^{\pm}=\sup\limits_{\Omega\cap \mathcal{P}_{2r}}\pm \frac {u(x)}{x_n},\qquad
\omega=m^++m^-=\osc{\Omega\cap \mathcal{P}_{2r}}\frac {u(x)}{x_n}.
$$
Since $u\big|_{\partial\Omega}=0$ we have $m^{\pm}\geqslant 0$. Therefore, at least one of the numbers
$m^{\pm}$ is not less than $\frac{\omega}{2}$, and both of the numbers $m^{\pm}$ are less than $\omega$.

\vspace{0.3cm}
Let $m^+\geqslant \frac{\omega}{2}$ for definiteness. Then we
consider the nonnegative function $ v(x)=m^+x_n-u(x)$ in $\Omega\cap
\mathcal{P}_{2r}$; (if $m^->\frac{\omega}{2}$ then we consider the function
$v(x)=m^-x_n+u(x)$).

\vspace{0.3cm}
Due to definition of $\delta$, for any sufficiently small $r>0$ we can find a point $x^*
\in \partial \mathcal{P}_{r}\cap
\partial\Omega$ such that
$x^*_n = r\delta (r)$. Without loss of generality we may assume that $x_1^*=r$ and $x^*_{\tau}=0$ for $\tau =2,\dots, n-1$.

Next we assign to $x^*$ a local orthogonal coordinate system $y_1,\dots,y_n$ such that
\begin{enumerate}
\item[(a)] $y_1$- axis is directed along the projection of the vector $(x_1^*,\dots,x_{n-1}^*)$ onto tangential hyperplane to $\partial\Omega$ at $x^*$;
\item[(b)] $y_2$, $\dots$, $y_{n-1}$-axes are parallel to $x_2$, $\dots$, $x_{n-1}$-axes, respectively;
\item[(c)] $y_n$-axis is directed inside $\Omega$.
\end{enumerate}

Due to the extremal property of $x^*$ the axes $y_1, \dots, y_{n-1}$ lie in the supporting hyperplane to $\partial\Omega$ at $x^*$. Moreover, if $x^*$ is a smooth point of $\partial\Omega$ then $y_n$ is directed along  the inward normal  to $\partial\Omega$. 

Setting $\gamma=\frac{\nu}{\sqrt{n-1}}$ we consider in $y$-coordinates the cylinder 
$$
\Pi:=\left\lbrace  
y\in \mathbb{R}^n: \left|y_1-\frac{r}{2}\right|<\frac{r}{2}, \ |y_{\tau}|<\frac{r}{2}, \
0<y_n<\frac{1}{2}\gamma r
\right\rbrace  ,
$$
and the ball $B_{\rho_0}(z^0)$ with
$
\rho_0=\frac{1}{8}\gamma r$ and $z^0=\left(\frac{r}{2}, 0,\dots, 0,\frac{1}{4}\gamma r\right)
$.

It should be emphasized that from now on, all considerations will be carried out in $x$-coordinates.

\begin{figure}[htbp]
\centering
\includegraphics[width=0.85\textwidth]{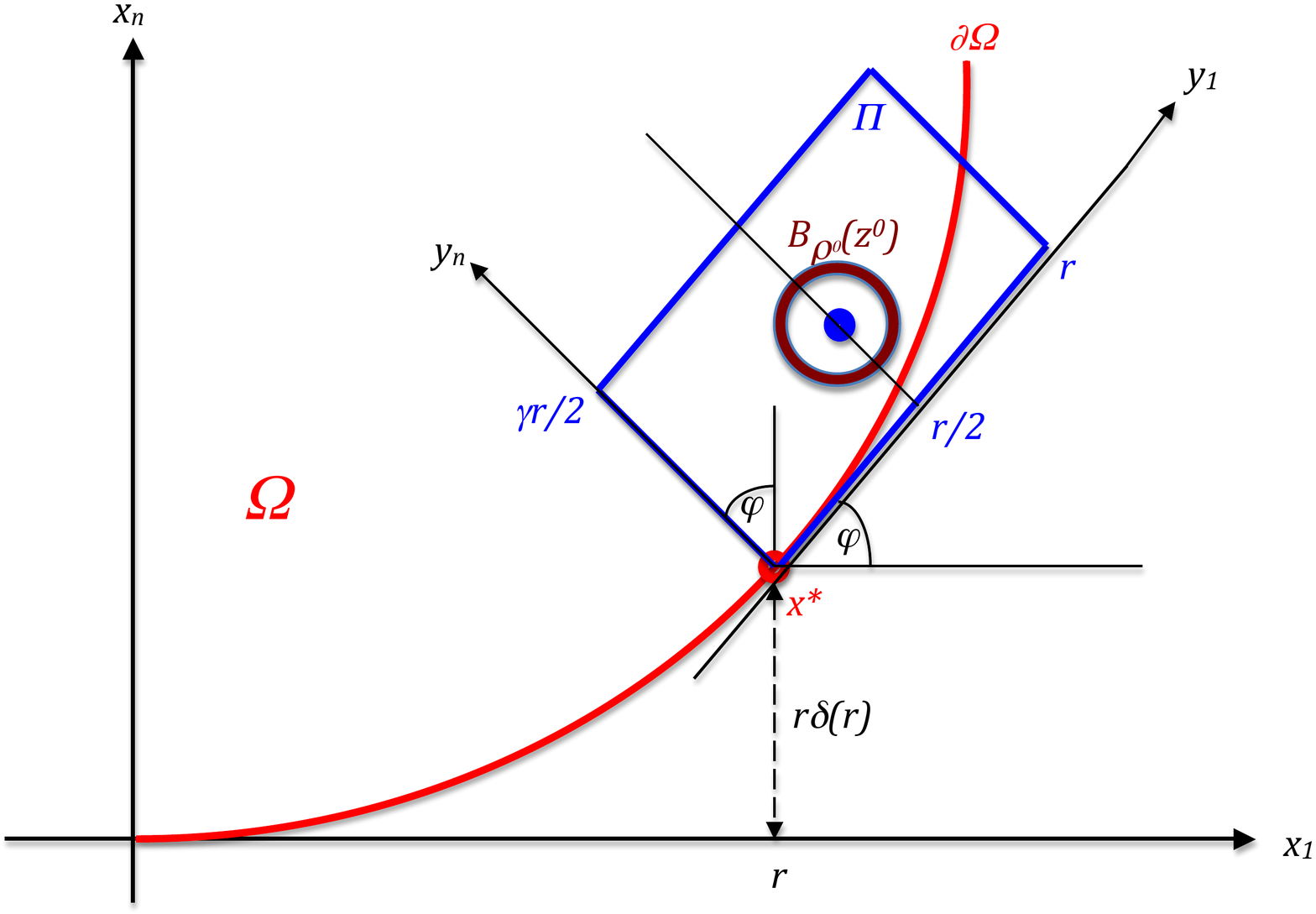}
\caption{Schematic view of $\Pi$ and $B_{\rho_0}(z^0)$.}
\label{fig-1}
\end{figure}

We claim that
\begin{equation} \label{claim-case1}
B_{\rho_0}(z^0)\subset \Omega.
\end{equation}
Indeed, assume that (\ref{claim-case1}) fails. Then there is a point $\hat{x} \in B_{\rho_0}(z^0)$ satisfying (in $x$-coordinates) the inequalities 
\begin{equation} \label{claim-2}
F(\hat{x}') \geqslant\hat{x}_n \geqslant z^0_n-\rho_0.
\end{equation}
Since $\hat{x}\in B_{\rho_0}(z^0)$ it is clear that $|\hat{x}'|\leqslant 2r$ and
$$
F(\hat{x}') \leqslant 2r \delta (2r). 
$$
On the other hand, denoting by $\varphi$ the angle between $x_n$- and $y_n$-axis (see Fig.~\ref{fig-1}) we conclude that
\begin{align*}
z^0_n-\rho_0 &=r \delta (r)+\frac{r}{2}\sin{\varphi}+\frac{\gamma r}{4}\cos{\varphi}-\frac{\gamma r}{8} 
\geqslant \frac{\gamma r}{8}\left( 2\cos{\varphi}-1\right). 
\end{align*}
Thus (\ref{claim-2}) is transformed into 
\begin{equation} \label{claim-3}
\gamma \left( 2\cos{\varphi}-1\right) \leqslant 16 \delta (2r).
\end{equation}

In view of (\ref{delta-to-zero}) and Lemma~\ref{Proposition}, one can choose  $R_0$ so small that 
$\delta_1(R_0)\leqslant 3/4$. It guarantees for all $r \leqslant R_0/2$ the inequalities
\begin{equation} \label{claim-4}
\cos{\varphi} =\frac{1}{\sqrt{1+\tan^2{\varphi}}}\geqslant \frac{1}{\sqrt{1+\delta_1^2(r)}}\geqslant \frac{1}{\sqrt{1+\delta_1^2(R_0)}} \geqslant \frac{4}{5}.
\end{equation}
Now, combining  (\ref{claim-4})  and (\ref{claim-3}) we get  a contradiction with relation (\ref{delta-to-zero}) provided $\delta (R_0)$ being small enough. The proof of (\ref{claim-case1}) is complete.
\paragraph{$\boxed{2.}$}
With (\ref{claim-case1}) at hands, we observe that
$$
\text{inf} \{x_n: x \in \Omega \cap \Pi \} \geqslant
r\delta (r).
$$
On the other hand, the condition $u=0$ for $x \in \partial\Omega
\cap \Pi $ gives the estimate
$$
v=m^+x_n \geqslant \frac{\omega}{2}x_n \quad \text{on}  \quad
\partial\Omega \cap \Pi.
$$
Hence, 
\begin{equation}
v \geqslant \frac{\omega}{2}r\delta (r)=:k_0\quad \text{on}\quad  \partial\Omega \cap
\Pi.
\label{2a}
\end{equation} 
So, we can apply Lemma~\ref{analog-lemma1-AU} to the function $v$ in cylinder $\Pi$. This
gives the estimate
$$
\inf\limits_{B_{\rho_0}(z^0)}  v \geqslant \left(   k_0 \big[  C_1-C_2 \|\mathbf{b}\|_{\mathcal{X}, \Omega \cap \mathcal{P}_{2r}} \big] -C_3\omega r\|b^n\|_{\mathcal{X}, \Omega \cap\mathcal{P}_{2r}}\right)_+
,
$$
where
$C_1$, $C_2$ 
and $C_3$ are 
the constants from Lemma~\ref{analog-lemma1-AU}. Decreasing $R_0$, if necessary, we may assume that $\|\mathbf{b}\|_{\mathcal{X}, \Omega \cap\mathcal{P}_{R_0}}\leqslant C_1/\left( 2C_2\right)$. Thus, we arrive~at
\begin{equation} \label{2}
\inf\limits_{B_{\rho_0}(z^0)}  v \geqslant \left( k_0\frac{C_1}{2}-C_3\omega r\|b^n\|_{\mathcal{X}, \Omega \cap\mathcal{P}_{2r}}\right)_+=:k_1. 
\end{equation}


Consider now  an arbitrary point
$\widetilde{z}=({\widetilde{z}}', r/4+\rho_0/8)$ such that
$|{\widetilde{z}}'|~\leqslant~\dfrac{r}{4}$. Observe also that $B_{\rho_0}(\widetilde{z}) \subset \Omega$, otherwise we get a contradiction with definition of $\delta (r)$.

We claim that 
\begin{equation}
\inf\limits_{B_{{\rho_0}/8}(\widetilde{z})}  v \geqslant
\big( k_0 \widetilde{C}_1 -\widetilde{C}_2   \omega r \|b^n\|_{\mathcal{X}, \Omega \cap\mathcal{P}_{2r}}\big)_+, 
\label{estimate-for-tilde-z}
\end{equation}
where 
$\widetilde{C}_1=\widetilde{C}_1(n,\nu)$,  whereas $\widetilde{C}_2$
is determined completely by   $n$, $\nu$, and  $\|\textbf{b}\|_{\mathcal{X}, \Omega}$.
Indeed, due to convexity of $\Omega$, for $l$ running from $1$ to a finite number
$\mathfrak{N}=\mathfrak{N}(n, \nu)$ chosen so that 
\begin{equation}\label{choice-of-N}
\frac{4}{3\rho_0}|z^0-\widetilde{z}| \leqslant \mathfrak{N}\leqslant \frac{2}{\rho_0}|z^0-\widetilde{z}|,
\end{equation}
and for points $z^{[l]}:=z^0-\frac{l}{\mathfrak{N}}(z^0-\widetilde{z})$ we
have $B_{\rho_0}(z^{[l]}) \subset \Omega$. It should be emphasized that  the lower and the upper bounds in (\ref{choice-of-N}) do not depend on~$r$.

 In view of (\ref{2}) we can 
compare in $\mathcal{B} (z^{[1]}, \rho_0/8, \rho_0)$ the function $v$ with the standard barrier function
$$
w(x)=k_1\,\frac{|x-z^{[1]}|^{-s}-\rho_0^{-s}}{(\rho_0/8)^{-s}-\rho_0^{-s}}.
$$
If $s = n\nu^{-2}$ then elementary calculation garantees the estimates
\begin{align*}
\mathcal{L} w &\leqslant |\mathbf{b}||Dw|\leqslant c(n, \nu)k_1|\mathbf{b}|\rho_0^{-1} \quad \text{in} \quad \mathcal{B} (z^{[1]}, \rho_0/8, \rho_0),\\
w(x)& = k_1 \leqslant v(x) \quad \text{on the sphere}\quad |x-z^{[1]}|=\frac{\rho_0}{8}\\
w(x)& =0\,\,\leqslant v(x)  \quad \text{on the sphere} \quad |x-z^{[1]}|=\rho_0.
\end{align*}
Application of the maximum principle (\ref{max-principle})   in $\mathcal{B} (z^{[1]}, \rho_0/8, \rho_0)$ to the difference $w-v$
gives us  the inequality
$$
v(x) \geqslant  \big( k_1\left[ w(x)-2cN_0 \|\mathbf{b}\|_{\mathcal{X}, \Omega \cap \mathcal{P}_{2r}}\right] - N_0 \frac{\gamma r}{4} \omega \|b^n\|_{\mathcal{X}, \Omega \cap \mathcal{P}_{2r}} \big)_+. 
$$
Since $B_{\rho_0/8}(z^{[2]}) \subset \mathcal{B}\left( z^{[1]}, \rho_0/8, 7\rho_0/8\right) $, the evident bound
$w \geqslant \theta (n,\nu)$ holds true in  $B_{\rho_0/8}(z^{[2]})$.

Decreasing $R_0$, if necessary, we ensure that $\|\mathbf{b}\|_{\mathcal{X}, \Omega \cap \mathcal{P}_{R_0}} \leqslant \left( 4cN_0\right)^{-1}\theta$. This implies
$$
\inf\limits_{B_{\rho_0/8}(z^{[2]})}v(x) \geqslant \left( \frac{ k_1\theta}{2}- N_0 \frac{\gamma r}{4}\omega \|b^n\|_{\mathcal{X}, \Omega \cap\mathcal{P}_{2r}} \right)_{+}=:k_2.
$$

Repeating this procedure for $\mathcal{B} (z^{[l]}, \rho_0/8, \rho_0)$ and $l=2, \dots, \mathfrak{N}$ we arrive at (\ref{estimate-for-tilde-z}) with 
$
\widetilde{C}_1=\left( \theta /2\right)^{\mathfrak{N}} 
$
and $\widetilde{C}_2=N_0\,\dfrac{\gamma}{4} \cdot\dfrac{1-\left( \theta /2\right)^{\mathfrak{N}} }{1-\left( \theta / 2\right) }$. 

Furthermore, it is clear that
$$
\left( k_0\widetilde{C}_1-\widetilde{C}_2 r\omega \|b^n\|_{\mathcal{X}, \Omega\cap\mathcal{P}_{2r}}\right)_+\geqslant 
\omega r \left( \frac{1}{2}\widetilde{C}_1\delta (r)-\widetilde{C}_2 \mathfrak{B}\sigma\left( r /\mathcal{R}_0\right) \right)_+,
$$
while inequalities (\ref{relation-1}) and (\ref{relation-sigma}) guarantee that
$$
\sigma\left( r /\mathcal{R}_0\right) \leqslant \frac{\mathcal{J}_{\sigma}(r)}{\mathcal{R}_0}.
$$
Decreasing again $R_0$ and
taking into account the assumption (\ref{relation-2}) and the above inequalities,
we can transform (\ref{estimate-for-tilde-z}) into the form
\begin{equation} \label{better-estimate-for-tilde-z}
\inf\limits_{B_{\rho_0/8}(\widetilde{z})} v\geqslant \frac{1}{4}\widetilde{C}_1\omega r\delta (r)=:\widetilde{k}.
\end{equation}

\paragraph{$\boxed{3.}$}
Now,
we take a small $\eta >0$, define the set 
$$
\mathcal{A}_{\eta}:=\mathcal{B}(\widetilde{z}, \rho_0/8, \widetilde{z}_n) \cap \Omega \cap \left\lbrace x\in \mathcal{P}_{\mathcal{R}_0} : F(x')+\eta <x_n<\mathcal{R}_0\right\rbrace 
$$
and introduce in $\mathcal{A}_{\eta}$ the barrier function 
$$
W(x)=\mu\widetilde{k}\,\frac{|x-\widetilde{z}|^{-s}-\left( \widetilde{z}_n\right)^{-s} }{\left( \rho_0/8\right)^{-s}-\left( \widetilde{z}_n\right)^{-s} },
$$
where $s=n\nu^{-2}$ and $0<\mu \leqslant 1$.

Notice that $D\left( Du\right) \in \mathcal{X}\left( \mathcal{A}_{\eta}\right)$. Using Lemma~\ref{operator-approximation} we construct the family  of operators $\mathcal{L}_{\varepsilon}$ satisfying $\|\mathcal{L}_{\varepsilon}u\|_{\mathcal{X}, \mathcal{A}_{\eta}}\rightarrow 0$ as $\varepsilon \rightarrow 0$.

Arguing in the spirit of the proof of Lemma~4.2 \cite{LU88},
we 
define 
 $v_1(x)$ and $v_2(x)$ as  solutions of the following problems:
$$
\left\lbrace \begin{aligned}
\mathcal{L}_{\varepsilon}v_1&=b^i_{\varepsilon}D_iW \ \, \text{in}\ \mathcal{A}_{\eta}\\
v_1&=v \quad \text{on}\ \partial\mathcal{A}_{\eta}
\end{aligned} \right., \qquad 
\left\lbrace \begin{aligned}
\mathcal{L}_{\varepsilon}v_2&=b^i_{\varepsilon}D_iW -b^n_{\varepsilon}m^+ \  \, \text{in}\ \mathcal{A}_{\eta}\\
v_2&=0 \quad \text{on}\ \partial\mathcal{A}_{\eta}
\end{aligned} \right..
$$

It is well known (see, for instance, \cite[Chapter 6]{Kr08}) that $D(Dv_1)$ and $D(Dv_2)$ belong to the space $BMO_{loc}\left( \mathcal{A}_{\eta}\right) $. Moreover, the John-Nirenberg theorem \cite{JN61}  (see also \cite[\S 4, Chapter 6]{D01}) implies that $D(Dv_i)$, $i=1,2$, belong to the Orlicz space $L_{\Phi,loc}(\mathcal{A}_{\eta})$ with $\Phi(\varsigma)=e^{\varsigma}-\varsigma-1$. So, taking into account the property (iii) we may conclude that $v_i \in \mathcal{W}^2_{\mathcal{X}, loc}\left( \mathcal{A}_{\eta}\right) $, $i=1,2$. \vspace{0.2cm}


Furthermore, in view of (\ref{better-estimate-for-tilde-z})  and the direct calculation, we have 
the inequalities
\begin{align*}
\mathcal{L}_{\varepsilon}W & \leqslant b^i_{\varepsilon}D_iW
\quad \text{in} \quad \mathcal{A}_{\eta},\\
W(x)&=\mu\widetilde{k} \leqslant v(x)=v_1(x) \quad \text{on the sphere}\quad |x-\widetilde{z}|=\frac{\rho_0}{8},\\
W(x)&=0\ \leqslant v(x)=v_1(x) \quad \ \text{on}\quad \partial \mathcal{A}_{\eta} \cap \left\lbrace x\in \mathbb{R}^n: |x-\widetilde{z}|=\widetilde{z}_n\right\rbrace .
\end{align*}
On the rest of $\partial \mathcal{A}_{\eta}$ we have $x_n=F(x')+\eta$ and, consequently, $\textit{dist}\left\lbrace x, \partial\Omega\right\rbrace\leqslant~\eta$. Since $u\in \mathcal{C}\left( \overline{\Omega}\right)$, the latter inequality  implies the estimate 
$u \leqslant H(\eta)$ there, and therefore, 
$$
v_1(x)=v(x)=m^+x_n-u\geqslant \dfrac{\omega }{2} x_n-H(\eta),
$$
where $H$ is a nonnegative function tending to zero as $\eta\rightarrow 0$.

In addition, it is easy to verify that
$$
W(x) \leqslant \mu N_6(n,\nu)\widetilde{C}_1 \omega \delta (r)x_n\quad \text{in}\  \overline{\mathcal{B}}(\widetilde{z}, \rho_0/8, \widetilde{z}_n).
$$
Choosing $\mu=\min \left\lbrace 1; \left( 2N_6\widetilde{C}_1\right)^{-1} \right\rbrace $,
we get
$$
v_1(x) \geqslant W(x)-H(\eta) \quad \text{on}\quad \partial\mathcal{A}_{\eta}.
$$
The maximum principle (\ref{max-principle}) applied to the difference $W-H(\eta)-v_1$ in $\mathcal{A}_{\eta}$ provides the inequality
$$
v_1(x) \geqslant W(x)-H(\eta) \geqslant \mu N_7(n, \nu) \widetilde{C}_1 \omega \delta (r) \left( \widetilde{z}_n-|x-\widetilde{z}|\right)- H(\eta). 
$$
It follows from the last inequality with
$x=(\widetilde{z}',x_n)\in \Omega$ and $0<x_n\leqslant \widetilde{z}_n-\rho_0/8= r/4$  that
\begin{equation} \label{v1-1}
v_1(\widetilde{z}',x_n) \geqslant N_8(n, \nu)\, \omega\, \delta (r) x_n-H(\eta).
\end{equation}


Next, we look for a majorant for $v_2$. With this aim in view, we extend the coefficients $a^{ij}_{\varepsilon}$ continuously and  and the coefficients $b^i_{\varepsilon}$ by zero to the whole annulus $\mathcal{B}(\widetilde{z}, \rho_0/8, \widetilde{z}_n)$, and denote by $\widetilde{v}_2(x)$ the solution of the problem
\begin{align*}
\mathcal{L}_{\varepsilon}\widetilde{v}_2&=\left\lbrace \begin{aligned}
\left( \mathcal{L}_{\varepsilon}v_2\right)_+ \quad &\text{in}\quad \mathcal{A}_{\eta},\\
 0 \qquad \quad &\text{in}\quad  \mathcal{B}(\widetilde{z}, \rho_0/8, \widetilde{z}_n) \setminus \mathcal{A}_{\eta};
\end{aligned} \right. \\
\widetilde{v}_2&=0 \quad \text{on} \quad \partial \mathcal{B}(\widetilde{z}, \rho_0/8, \widetilde{z}_n).
\end{align*}
The maximum principle guarantees  
\begin{equation} \label{v2-1}
v_2 \leqslant \widetilde{v}_2  \quad \text{in} \quad \mathcal{A}_{\eta}.
\end{equation}

Direct computations show that for $\rho \leqslant r/4$ the barrier function $W$ satisfies in the set $\mathcal{E}_{\rho}:=\mathcal{P}_{\rho}(\widetilde{z}',0) \cap \mathcal{B}(\widetilde{z}, \rho_0/8, \widetilde{z}_n)$ 
the following inequalities 
\begin{align*}
|D_nW|\leqslant |DW| & \leqslant N_9(n, \nu)\,\mu\, \frac{\widetilde{k}}{r}\leqslant N_9\,\omega\,\delta (r), \\
|D'W| &\leqslant N_9 \mu \frac{\widetilde{k}\rho}{r^2} \leqslant N_9\,\omega\,\frac{\delta (r)\rho}{r}.
\end{align*}
So, in view of (\ref{app-estimate-b}) and (\ref{delta-to-zero}), we have for all $\rho \leqslant r/4$ 
the bounds
\begin{align*}
\|\left( \mathcal{L}_{\varepsilon}\widetilde{v}_2\right)_{+} \|_{\mathcal{X}, \mathcal{E}_{\rho}} &\leqslant 
\|b^n\|_{\mathcal{X}, \mathcal{E}_{\rho}}\bigg( m^++\|D_nW\|_{\infty, \mathcal{E}_{\rho}}\bigg) +\|\mathbf{b}'\|_{\mathcal{X}, \mathcal{E}_{\rho}}\|D'W\|_{\infty, \mathcal{E}_{\rho}}\\
&\leqslant N_{10} (n, \nu)\,\omega \, \left[\mathfrak{B} \sigma \left( \rho/\mathcal{R}_0\right)+\frac{\delta (r)}{r} \rho \|\mathbf{b}'\|_{\mathcal{X}, \mathcal{A}_{\eta}}  \right].
\end{align*}
Since the function $\rho \mapsto \left[\mathfrak{B} \sigma \left( \rho/\mathcal{R}_0\right)+\frac{\delta (r)}{r} \rho \|\mathbf{b}'\|_{\mathcal{X}, \mathcal{A}_{\eta}}  \right]$ satisfies the Dini condition at zero, there exist the uniquely defined function $\sigma_1 \in \mathcal{D}_1$ and a constant $\mathfrak{B}_1$ such that
$$
\mathfrak{B} \sigma \left( \rho/\mathcal{R}_0\right)+\frac{\delta (r)}{r} \rho \|\mathbf{b}'\|_{\mathcal{X}, \mathcal{A}_{\eta}}  = \mathfrak{B}_1 \sigma_1\left( 4\rho /r\right). 
$$

Thus, we may
apply Lemma~\ref{lemma2.3-LU88} to the function $\widetilde{v}_2$. It gives for $ \rho =r/4$
 the estimate
\begin{equation} \label{v2-2}
\begin{aligned}
\sup\limits_{0<x_n<r/4}\frac{\widetilde{v}_2(\widetilde{z}',x_n)}{x_n}
&\leqslant C_4 \left( \left( r/4\right) ^{-1}\sup\limits_{\mathcal{E}_{r/4}}\widetilde{v}_2+N_{10}\omega \mathfrak{B}_1\mathcal{J}_{\sigma_1}\left( 1\right) \right). 
\end{aligned}
\end{equation}
It is easy to see that
$$
\mathfrak{B}_1\mathcal{J}_{\sigma_1}(1)= \mathfrak{B}\mathcal{J}_{\sigma}\left( \frac{r}{4 \mathcal{R}_0}\right) +
\frac{\delta (r)}{4}\|\mathbf{b}'\|_{\mathcal{X}, \mathcal{A}_{\eta}}. 
$$
Furthermore,
applying (\ref{max-principle}) to $\widetilde{v}_2$ and to the operator $\mathcal{L}_{\varepsilon}$ in $\mathcal{B}(\widetilde{z}, \rho_0/8, \widetilde{z}_n)$, we obtain
$$
\sup\limits_{\mathcal{E}_{r/4}}\widetilde{v}_2 \leqslant \sup\limits_{\mathcal{B}(\widetilde{z}, \rho_0/8, \widetilde{z}_n)} \widetilde{v}_2 \leqslant N_{11}(n, \nu, \|\mathbf{b}\|_{\mathcal{X}, \Omega})\,\omega
 r \left[ \mathfrak{B}\sigma \left( \frac{r}{\mathcal{R}_0}\right) +\delta (r) \|\mathbf{b}'\|_{\mathcal{X}, \mathcal{A}_{\eta}}\right]. 
$$
Substitution of the above estimates in (\ref{v2-2}) and having regard to (\ref{relation-1})  provide
\begin{equation} \label{v2-3}
\sup\limits_{0<x_n<r/4}\frac{\widetilde{v}_2(\widetilde{z}',x_n)}{x_n}\leqslant N_{12} \,\omega \left[ 
\mathfrak{B}\mathcal{J}_{\sigma}\left( \frac{r}{\mathcal{R}_0}\right) +\delta (r) \|\mathbf{b}'\|_{\mathcal{X}, \mathcal{A}_{\eta}}\right],
\end{equation}
where the constant $N_{12}$ depends only on $n$, $\nu$ and $\|\mathbf{b}\|_{\mathcal{X}, \Omega}$. \vspace{0.1cm}

Taking into account  the inequality (\ref{relation-J-sigma}), the assumption (\ref{relation-2}), and  the evident relation $\|\mathbf{b}'\|_{\mathcal{X}, \mathcal{A}}=o(1)$ as $r\rightarrow 0$, we decrease $R_0$ such that the property
\begin{equation} \label{relation-3}
\left[  
\mathfrak{B}\mathcal{J}_{\sigma}\left( \frac{r}{\mathcal{R}_0}\right) +\delta (r) \|\mathbf{b}'\|_{\mathcal{X}, \mathcal{A}_{\eta}}\right] \leqslant \frac{N_8}{2N_{12}}\delta (r)
\end{equation}
holds true for all $r \leqslant R_0$.

Finally, combining (\ref{v1-1})-(\ref{v2-1}) with (\ref{v2-3})-(\ref{relation-3}) 
we arrive   at the estimate
\begin{equation} \label{v-final}
v_1(\widetilde{z}', x_n)-v_2(\widetilde{z}', x_n) \geqslant \frac{N_8}{2}\omega \delta (r)x_n -H(\eta)
\end{equation}
for $r \leqslant R_0$ and $x=(\widetilde{z}',x_n) \in \Omega$ with $x_n \in [F(\widetilde{z}')+\eta, r/4]$.

Considering in $\mathcal{A}_{\eta}$ the function $v_3(x)=v(x)-v_1(x)+v_2(x)$  one can easily see that 
$$
\mathcal{L}_{\varepsilon}v_3=-\mathcal{L}_{\varepsilon}u \rightarrow 0 \quad \text{in}\quad \mathcal{X}\left( \mathcal{A}_{\eta}\right) \quad \text{as}\quad \varepsilon \rightarrow 0.
$$
In addition, $v_3=0$ on $\partial\mathcal{A}_{\eta}$. Applying the maximum principle (\ref{max-principle}) to $\pm v_3$ and to the operator $\mathcal{L}_{\varepsilon}$ we obtain that the difference $v_1(x)-v_2(x)$ converges to $v(x)$ uniformly in $\mathcal{A}_{\eta}$. Therefore,
passing in (\ref{v-final}) first to the limit as $\varepsilon \to 0$ and then as $\eta \rightarrow 0$, we get 
\begin{equation} \label{v-2final}
\frac{v(x)}{x_n} \geqslant \frac{N_8}{2}\omega \delta (r).
\end{equation}
for $r \leqslant R_0$ and $x=(\widetilde{z}',x_n) \in \Omega$ with $x_n \in [F(\widetilde{z}'), r/4]$.


Since $\widetilde{z}'$ can be chosen arbitrarily with only $|\widetilde{z}'|\leqslant \dfrac{r}{4}$, the estimate (\ref{v-2final}) gives (\ref{1}) with $\varkappa= N_8/2 $. 
\end{proof}

\vspace{0.2cm}
\begin{thm}[\textbf{Main Theorem}]
Let the assumptions of Theorem~\ref{osc-estimate} hold, and let $\delta (r)=\max\limits_{|x'|\leqslant r}\frac{F(x')}{|x'|}$ do not satisfy the Dini condition at zero.

Then for any function $u$ satisfying (\ref{main-assumptions-on-u}) the equality
$$
\frac{\partial u}{\partial \mathbf{n}}(0)=0
$$
holds true.
\end{thm}

\begin{proof}
Consider the sequence $r_k=8^{-k}R_0$, $k\geqslant 0$, where $R_0$ is the constant from Theorem~\ref{osc-estimate}.

Application of
Theorem~\ref{osc-estimate}  to $u$ guarantees for $k \geqslant 0$ the following inequalities
$$
\osc{\Omega\cap \mathcal{P}_{r_{k+1}}}\frac {u(x)}{x_n}\leqslant \left( 1-\varkappa \delta (r_k/2)\right) \osc{\Omega\cap \mathcal{P}_{r_k}}\frac {u(x)}{x_n} \leqslant \osc{\Omega\cap \mathcal{P}_{R_0}}\frac {u(x)}{x_n} \cdot
\prod\limits_{j=0}^k \left( 1-\varkappa \delta (r_j/2)\right).
$$
Since 
$$
\sum\limits_{j=0}^{\infty}\ln \left( 1-\varkappa \delta  (r_j/2)\right) \asymp -\sum\limits_{j=0}^{\infty}\delta (r_j/2) \asymp -\int\limits_{0}^{r_0}\frac{\delta (r)}{r}dr =-\infty,
$$
we have
$$
\prod\limits_{j=0}^{k} \left( 1-\varkappa \delta (r_j/2)\right)\rightarrow 0 \quad \text{as}\quad k\rightarrow \infty.
$$
We recall also that Lemma~\ref{lemma2.3-LU88} implies the finiteness of the quantity $\osc{\Omega\cap \mathcal{P}_{R_0}}\dfrac {u(x)}{x_n}$. 

Thus, taking into account that $u\big|_{\partial\Omega \cap \mathcal{P}_{\mathcal{R}_0}}=0$ we get

$$
\left|\frac{\partial u}{\partial \mathbf{n}}(0)\right|=\left|\lim\limits_{x_n\rightarrow 0}\frac{u(0,x_n)}{x_n}\right|\leqslant \lim\limits_{k \to \infty}\left| \osc{\Omega\cap \mathcal{P}_{r_k}}\dfrac {u(x)}{x_n}\right|=0,
$$
and complete the proof.
\end{proof}

\section*{Acknowledgement}

The authors would like to thank the anonymous referee for essential comments.

This work was supported by the Russian Foundation of Basic Research (RFBR) through the grant 
15-01-07650  and by the St. Petersburg State University grant 6.38.670.2013.






\bibliography{Bibliography(dcd)}

\newcommand{\etalchar}[1]{$^{#1}$}
\def\cprime{$'$}
\providecommand{\bysame}{\leavevmode\hbox to3em{\hrulefill}\thinspace}
\providecommand{\MR}{\relax\ifhmode\unskip\space\fi MR }
\providecommand{\MRhref}[2]{%
  \href{http://www.ams.org/mathscinet-getitem?mr=#1}{#2}
}
\providecommand{\href}[2]{#2}
\begin{thebibliography}{ABM{\etalchar{+}}11}

\bibitem[ABM{\etalchar{+}}11]{AZ11}
R.~Alvarado, D.~Brigham, V.~Maz'ya, M.~Mitrea, and E.~Ziad{\'e}, \emph{On the
  regularity of domains satisfying a uniform hour-glass condition and a sharp
  version of the {H}opf-{O}leinik boundary point principle}, Probl. Mat. Anal.
  \textbf{57} (2011), 3--68 [Russian], English transl. in J. Math. Sci. (N.Y.)
  \textbf{176}, no. 3 (2011), 281-360.

\bibitem[Ale60]{Al60}
A.~D. Aleksandrov, \emph{Certain estimates concerning the {D}irichlet problem},
  Dokl. Akad. Nauk SSSR \textbf{134} (1960), no.~5, 1000--1004 [Russian],
  English transl. in Soviet Math. Dokl. \textbf{1} (1961), 1151-1154.
  \MR{0147776 (26 \#5290)}

\bibitem[Ale63]{Al63}
\bysame, \emph{Uniqueness conditions and estimates for a solution of the
  {D}irichlet problem}, Vest. Leningr. Univ. Ser. Mat. Mekh. Astron.
  \textbf{18} (1963), no.~3, 5--29 [Russian]. \MR{0164135 (29 \#1434)}

\bibitem[Alv11]{Alv11}
R.~Alvarado, \emph{Topics in harmonic analysis and partial differential
  equations: extension theorems and geometric maximum principles}, Masters
  {T}hesis, University of Missouri, 2011.

\bibitem[AN00]{AN01}
D.~E. Apushkinskaya and A.~I. Nazarov, \emph{The {D}irichlet problem for
  quasilinear elliptic equations in domains with smooth closed edges}, Probl.
  Mat. Anal. \textbf{21} (2000), 3--29 [Russian], English transl. in J. Math.
  Sci. (N.Y.) \textbf{105}, no. 5 (2001), 2299-2318. \MR{1855435 (2002m:35061)}

\bibitem[AU95]{AU95}
D.~E. Apushkinskaya and N.~N. Ural{\cprime}tseva, \emph{On the behavior of the
  free boundary near the boundary of the domain}, Zap. Nauchn. Sem.
  S.-Peterburg. Otdel. Mat. Inst. Steklov. (POMI) \textbf{221} (1995), 5--19
  [Russian], English transl. in J. Math. Sci. (N.Y.) \textbf{87}, no. 2 (1997),
  3267-3276. \MR{1359745 (96m:35340)}

\bibitem[Bak61]{B61}
I.~Ya. Bakel'man, \emph{On the theory of quasilinear elliptic equations}, Sib.
  Mat. Zh. \textbf{2} (1961), no.~179-186 [Russian]. \MR{0126604 (23 A3900)}

\bibitem[Duo01]{D01}
J.~Duoandikoetxea, \emph{Fourier analysis}, Graduate Studies in Mathematics,
  vol.~29, American Mathematical Society, Providence, RI, 2001.

\bibitem[Gir32]{G32}
G.~Giraud, \emph{Generalisation des probl\`emes sur les operations du type
  elliptique}, Bull. des Sciences Math. \textbf{56} (1932), 316--352.

\bibitem[Gir33]{G33}
\bysame, \emph{Probl\`emes de valeurs \`a la fronti\`ere relatifs \`a certaines
  donn\'ees discontinues}, Bull. Soc. Math. France \textbf{61} (1933), 1--54.
  \MR{1504997}

\bibitem[Him70]{Him70}
B.~N. Him{\v{c}}enko, \emph{On the behavior of solutions of elliptic equations
  near the boundary of a domain of type {$A^{(1)}$}}, Dokl. Akad. Nauk SSSR
  \textbf{193} (1970), 304--305 [Russian], English transl. in Soviet Math.
  Dokl. \textbf{11} (1970), 943-944. \MR{0273190 (42 \#8071)}

\bibitem[Hop52]{H52}
E.~Hopf, \emph{A remark on linear elliptic differential equations of second
  order}, Proc. Amer. Math. Soc. \textbf{3} (1952), 791--793. \MR{0050126
  (14,280b)}

\bibitem[JN61]{JN61}
F.~John and L.~Nirenberg, \emph{On functions of bounded mean oscillation},
  Comm. Pure Appl. Math. \textbf{14} (1961), no.~415--426.

\bibitem[KA82]{KA82}
L.~V. Kantorovich and G.~P. Akilov, \emph{Functional analysis}, second ed.,
  Translated from the Russian by H.L. Silcock, Pergamon Press, Oxford-Elmsford,
  N.Y., 1982.

\bibitem[KH75]{KaHim75}
L.~I. Kamynin and B.~N. Him{\v{c}}enko, \emph{Theorems of {G}iraud type for a
  second order elliptic operator that is weakly degenerate near the boundary},
  Dokl. Akad. Nauk SSSR \textbf{224} (1975), no.~4, 752--755 [Russian], English
  transl. in Soviet Math. Dokl. \textbf{16}, No. 5 (1975), 1287-1291.

\bibitem[KH77]{KaHim77}
\bysame, \emph{Theorems of {G}iraud type for second order equations with a
  weakly degenerate non-negative characteristic part}, Sibirsk. Mat. \v Z.
  \textbf{18} (1977), no.~1, 103--121 [Russian], English transl. in Sib. Math.
  J. \textbf{18} (1977), 76-91.

\bibitem[KL37]{KeLa37}
M.~V. Keldysh and M.~A. Lavrent'ev, \emph{On the uniqueness of the {N}eumann
  problem}, Dokl. Akad. Nauk SSSR \textbf{16} (1937), no.~3, 151--152
  [Russian].

\bibitem[Kor01]{Ko01}
A.~Korn, \emph{Lehrbuch der potentialtheorie. {II}. allgemeine theorie des
  logarithmischen potentials und der potentialfunctionen in der ebene}, Berlin:
  F. D{\"u}mmler, 1901.

\bibitem[Kry08]{Kr08}
N.~V. Krylov, \emph{Lectures on elliptic and parabolic equations in {S}obolev
  spaces}, Graduate Studies in Mathematics, vol.~96, American Mathematical
  Society, Providence, RI, 2008.

\bibitem[Lic24]{Lich24}
L.~Lichtenstein, \emph{Neue {B}eitr{\"a}ge zur {T}heorie der linearen
  partiellen {D}ifferentialgleichungen zweiter {O}rdnung vom elliptischen
  {T}ypus}, Math. Zeitschr. \textbf{20} (1924), 194--212.

\bibitem[Lie85]{Lie85}
G.~M. Lieberman, \emph{Regularized distance and its applications}, Pacific J.
  Math. \textbf{117} (1985), no.~2, 329--352. \MR{779924 (87j:35101)}

\bibitem[LU85]{LU85}
O.~A. Ladyzhenskaya and N.~N. Ural{\cprime}tseva, \emph{Estimates of the
  {H}\"older constant for functions satisfying a uniformly elliptic or
  uniformly parabolic quasilinear inequality with unbounded coefficients}, Zap.
  Nauchn. Sem. Leningrad. Otdel. Mat. Inst. Steklov. (LOMI) \textbf{147}
  (1985), 72--94, 204, Boundary value problems of mathematical physics and
  related problems in the theory of functions, No. 17. \MR{MR821476
  (87e:35090)}

\bibitem[LU88]{LU88}
O.~A. Ladyzhenskaya and N.~N. Ural{\cprime}tseva, \emph{Estimates on the
  boundary of a domain for the first derivatives of functions satisfying an
  elliptic and parabolic inequality}, Trudy MIAN SSSR \textbf{179} (1988),
  102--125 [Russian], English transl. in Proc. Steklov Inst. Math., Issue 2
  (1989), 109-135. \MR{0964915 (89k:35083)}

\bibitem[MS15]{MSh15}
H.~Mikayelyan and H.~Shahgholian, \emph{Hopf's lemma for a class of
  singular/degenerate {PDE}'s}, Ann. Acad. Sci. Fenn. \textbf{40} (2015),
  475--484.

\bibitem[Nad83]{N83}
N.~S. Nadirashvili, \emph{On the question of the uniqueness of the solution of
  the second boundary value problem for second-order elliptic equations}, Mat.
  Sb. (N.S.) \textbf{122 (164)} (1983), no.~3, 341--359 [Russian], English
  transl. in Math. USSR - Sbornik, \textbf{50}, No. 2 (1985), 25-341.
  \MR{721393 (85f:35070)}

\bibitem[Naz01]{Naz01}
A.~I. Nazarov, \emph{Estimates for the maximum of solutions of elliptic and
  parabolic equations in terms of weighted norms of the right-hand side},
  Algebra i Analiz \textbf{13} (2001), no.~2, 151--164 [Russian], English
  transl. in St. Petersburg Math. J., \textbf{13}, No. 2 (2002), 269-279.
  \MR{1834864 (2002j:35049)}

\bibitem[Naz05]{Naz07}
\bysame, \emph{The maximum principle of {A.}{D.} {A}leksandrov}, Sovrem. Mat.
  Prilozh. (2005), no.~29, 129--145 [Russian], English transl. in J. Math. Sci.
  (N.Y.) \textbf{142}, no. 3 (2007), 2154-2171. \MR{2465040 (2011b:35051)}

\bibitem[Naz12]{Naz12}
\bysame, \emph{A centennial of the {Z}aremba-{H}opf-{O}leinik lemma}, SIAM J.
  Math. Anal. \textbf{44} (2012), no.~1, 437--453. \MR{2888295}

\bibitem[Neu88]{Neu88}
C.~Neumann, \emph{{{\"U}}ber die {M}ethode des arithmetischen {M}ittels},
  Abhand. der K{\"o}nigl. S{\"a}chsischen Ges. der Wissenschaften. Leipzig
  \textbf{10} (1888), 662--702.

\bibitem[NU09]{NU09}
A.~I. Nazarov and N.~N. Uraltseva, \emph{Qualitative properties of solutions to
  elliptic and parabolic equations with unbounded lower-order coefficients},
  preprint 2009-05, St. Petersburg Math. Soc. El. Prepr. Archive, 2009.

\bibitem[Ole52]{O52}
O.~A. Ole{\u\i}nik, \emph{On properties of solutions of certain boundary
  problems for equations of elliptic type}, Mat. Sb. (N.S.) \textbf{30 (72)}
  (1952), 695--702. \MR{0050125 (14,280a)}

\bibitem[Saf08]{S08}
M.~V. Safonov, \emph{Boundary estimates for positive solutions to second order
  elliptic equations}, preprint, http://arxiv.org/abs/0810.0522, 2008.

\bibitem[Saf10]{S10}
\bysame, \emph{Non-divergence elliptic equations of second order with unbounded
  drift}, Nonlinear partial differential equations and related topics, vol.
  229, Amer. Math. Soc. Transl. Ser. 2, no. 211-232, Amer. Math. Soc.,
  Providence, RI, 2010, pp.~211--232.

\bibitem[Swe97]{Sw97}
G.~Sweers, \emph{Hopf's lemma and two dimensional domains with corners}, Rend.
  Istit. Mat. Univ. Trieste \textbf{XXVIII} (1997), 383--419.

\bibitem[Tol83]{T83}
P.~Tolksdorf, \emph{On the {D}irichlet problem for quasilinear equations in
  domains with conical boundary points}, Comm. Partial Differential Equations
  \textbf{8} (1983), no.~7, 773--817.

\bibitem[Ura96]{U96}
N.~N. Ural{\cprime}tseva, \emph{{$C\sp 1$} regularity of the boundary of a
  noncoincident set in a problem with an obstacle}, Algebra i Analiz \textbf{8}
  (1996), no.~2, 205--221, English transl. in\textit{ J. Math. Sci. (N.Y.)},
  \textbf{87} (2) (1997), 3267-3276. \MR{1392033 (97m:35105)}

\bibitem[Wid67]{W67}
K.-O. Widman, \emph{Inequalities for the {G}reen function and boundary
  continuity of the gradient of solutions of elliptic differential equations},
  Math. Scand. \textbf{21} (1967), 17--37. \MR{0239264 (39 \#621)}

\bibitem[Zar10]{Z10}
S.~Zaremba, \emph{Sur un probl\`eme mixte relatif \`a l' \'equation de
  {L}aplace}, Bull. Acad. Sci. Cracovie. Cl. Sci. Math. Nat. Ser. A (1910),
  313--344.

\end{thebibliography}

\Addresses

\end{document}